\documentclass[10pt]{amsart}
\usepackage{graphicx}
\usepackage{amsmath}
\usepackage{amsfonts}
\usepackage{amsthm} 
\usepackage{verbatim} 
\usepackage{amssymb}
\usepackage{epstopdf}

\numberwithin{equation}{section}

\newtheorem{theorem}[equation]{Theorem}
\newtheorem{lemma}[equation]{Lemma}
\newtheorem{corollary}[equation]{Corollary}

\newtheorem{proposition}[equation]{Proposition}

\newcommand{\Div}{{\mathrm{Div}}}
\newcommand{\PDiv}{{\mathrm{PDiv}}}

\newcommand{\p}{\mbox{.}}
\newcommand{\Z}{\mathbb{Z}}

\newcommand{\R}{\mathbb{R}}
\newcommand{\Q}{\mathbb{Q}}
\newcommand{\st}{\mbox{ } | \mbox{ }}

\title{A Riemann-Roch theorem for edge-weighted graphs}
\author{Rodney James}
\address{Department of Mathematics\\Colorado State University\\Fort Collins, CO 80523}
\curraddr{Department of Mathematical and Statistical Sciences\\University of Colorado\\Denver, CO 80217-3364}
\author{Rick Miranda}
\address{Department of Mathematics\\Colorado State University\\Fort Collins, CO 80523}

\begin{document}

\begin{abstract}
We prove a Riemann-Roch theorem for real divisors on edge-weighted graphs over the reals,
extending the result of Baker and Norine for integral divisors on graphs with multiple edges.
\end{abstract}
\maketitle


\section{Introduction}
The purpose of this article is to prove a Riemann-Roch theorem
for edge-weighted graphs,
inspired by (and extending)
the theorem of Baker and Norine (see \cite{BN}).
In that context, graphs without loops but with multiple edges are considered.
We consider the existence of multiple edges to be equivalent
to assigning to each pair of vertices an integral weight
which records the number of edges between them.
In our setting we consider arbitrary positive real numbers as edge weights.
This variation forces several interesting adjustments to be made to the theory.

Let $R$ be a subring of the real numbers $\R$.
An \emph{$R$-graph} $G$ is a finite connected graph
(without loops or multiple edges)
where each edge is assigned a weight, which is a positive element of $R$.
If we let the $n$ vertices of $G$ be $\{v_1,\ldots,v_n\}$,
we will denote by $p_{ij} = p_{ji}$ the weight of the edge joining $v_i$ and $v_j$.
If there is no edge connecting $v_i$ and $v_j$, we set $p_{ij}=p_{ji}=0$.

We define the \emph{degree} of a vertex $v_j$ of $G$ to be
the sum of the weights of the edges incident to it:
\[
\deg(v_j) = \sum_{i\neq j} p_{ij}.
\]

The \emph{edge-weighted Laplacian} matrix $P$ of $G$ is the symmetric $n \times n$ matrix defined by 
\[
(P)_{ij} =
\left\{
\begin{array}{ll}
-p_{ij} & \mbox{if } i \ne j \\
\deg(v_j) & \mbox{if } i = j.
\end{array}
\right.
\]
Note that if each $p_{ij} \in \{0,1\}$, $P$ is the Laplacian matrix of a regular graph;
as is the case for the regular graph Laplacian, $P$ is semi-positive definite with kernel
generated by $(1,1, \ldots, 1)$.

The \emph{genus} of $G$ is defined as
\[
g = \sum_{i<j} p_{ij} - n + 1,
\]
which allows $g$ to be negative when the $p_{ij}$ are sufficiently small.

An \emph{$R$-divisor} $D$ on $G$ is a formal sum
\[
D = \sum_{i=1}^n d_i \cdot v_i
\]
where each $d_i \in R$;
the divisors form a free $R$-module $\Div(G)$ of rank $n$.
We write $D_1 \geq D_2$ if the inequality holds at each vertex;
for a constant $c$, we write $D \geq c$
(respectively $D > c$) if $d_i \geq c$ 
(respectively $d_i > c$)
for each $i$.

The \emph{degree} of a divisor $D$ is 
\[
\deg(D)=\sum_{i=1}^n d_i
\] 
and the \emph{ceiling} of $D$ is the divisor
\[
\lceil D \rceil = \sum_{i=1}^n \lceil d_i \rceil \cdot v_i.
\]
The degree map is a homomorphism from $\Div(G)$ to $R$,
and the kernel $\Div_0(G)$ of divisors of degree zero
is a free $R$-module of rank $n-1$.

Let $H_j = \deg(v_j) \cdot v_j - \sum_{i\neq j} p_{ij} \cdot v_i$,
and set $\PDiv(G) = \left\{ \sum_{i=1}^n c_i H_i \st c_i \in \Z \right\}$
to be the free $\Z$-module generated by the $H_j$.
(Note that the $H_j$ divisors correspond to the columns of the matrix $P$.)
If $G$ is connected, $\PDiv(G)$ has rank $n-1$.

For two divisors $D, D' \in \Div(G)$,
we say that $D$ is \emph{linearly equivalent} to $D'$,
and write $D \sim D'$,
if and only if $D - D' \in \PDiv(G)$.

The \emph{linear system} associated with a divisor $D$ is
\begin{eqnarray*}
|D| &=& \{ D' \in \Div(G) \st D \sim D'
\mbox{ with } \lceil D' \rceil \ge 0  \}\\
&=& \{ D' \in \Div(G) \st D \sim D'
\mbox{ with } D' > -1 \}.
\end{eqnarray*}
We note that linearly equivalent divisors have the same linear system.
The use of the ceiling divisor in the definition above is the critical difference
between this theory and the integral theory developed by Baker and Norine \cite{BN}.
The following lemma gives a condition for $|D|=\emptyset$.
\begin{lemma} \label{DEmptyLemma}
If $D \le -1$, then $|D|=\emptyset$.
\end{lemma}
\begin{proof}
Suppose that $D \le -1$ and thus $\deg(D) \le -n$.  
If $|D| \ne \emptyset$, there is a $H \in \PDiv(G)$ such that $H+D > -1$,
and thus $\deg(H+D) > -n$.
Since $\deg(H)=0$, $\deg(H+D) = \deg(H) + \deg(D) = \deg(D)$,
hence we must have $|D| = \emptyset$.
\end{proof}

The essence of the Riemann-Roch theorem, for divisors on algebraic curves,
is to notice that the linear system corresponds to a vector space of rational functions,
and to relate the dimensions of two such vector spaces.
In our context we do not have vector spaces;
so we measure the size of the linear system in a different way (as do Baker and Norine).

Define the $h^0$ of an $R$-divisor $D= \sum_{i=1}^n d_i \cdot v_i$
\[
h^0(D) = \min\{ \deg(E) \st E \mbox{ is an $R$-divisor, }
E \geq 0 \mbox{ and } |D-E| = \emptyset \}.
\]
Note that $h^0(D) \ge 0$, with equality if and only if $|D|=\emptyset$
(since $E \ge 0$, $\deg(E)=0$ if and only if $E=0$ and thus $|D|=\emptyset$).   
We can find an upper bound for $h^0(D)$ as follows:  
set $E = \sum_{i=1}^{n} \max\{d_{i}+1,0\} \cdot v_{i}$, then $D-E \le -1$
and by Lemma~\ref{DEmptyLemma}, $|D-E|=\emptyset$ and thus
$h^0(D) \le \sum_{i=1}^n \max\{d_i+1,0\}$.
Since $h^{0}$ is defined to be the minimum degree of an $R$-divisor, 
$h^{0} \in R$; however, we will show that $h^{0}(D)$ does not depend on the change of $R$.

The \emph{canonical divisor} of $G$ is defined as 
\[
K = \sum (\deg(v_i)-2) \cdot v_i.
\]
Note that $\deg(K)=2g-2$.

The Riemann-Roch result that we will prove can now be stated.

\begin{theorem}\label{RRTheorem}
Let $G$ be a connected $R$-graph as above,
and let $D$ be an $R$-divisor on $G$.  Then
\[
h^0(D) - h^0(K-D) = \deg(D) + 1 - g.
\]
\end{theorem}
Since $h^0(K-D)\ge 0$, the classical Riemann inequality $h^0(D) \ge \deg(D)+1-g$ holds.
The results of Baker and Norine (see \cite{BN})
are exactly that the above theorem holds in the case of the subring $R = \Z$.
Our proof depends on the Baker-Norine Theorem in a critical way;
it would be interesting to provide an independent proof.

In \cite{GK} and \cite{MZ},
a Riemann-Roch theorem is proved for metric graphs with integral divisors;
these results differ from the present result in two fundamental ways.
First, our edge weights $p_{ij}$
and the coefficients of the divisors
are elements of the ring $R$.
Second, the genus $g$ is in $R$ for the present result,
whereas in \cite{GK} and \cite{MZ}, $g$ is a nonnegative integer.

We close this section with an example.
Consider the $R$-graph $G$ with two vertices and edge weight $p>0$.
For convenience, we will write the divisor $a \cdot v_1 + b \cdot v_2$
as the ordered pair $(a,b)$.
The principal divisors are
$\PDiv(G)=\{ (np,-np) \st n \in \Z\}$,
and $K=(p-2,p-2)$, with $g=p-1$.
Note that if $p<1$, we have $g<0$.

For $(a,b) \in \Div(G)$,
the linear system $|(a,b)|$ can be written as
\begin{eqnarray*}
|(a,b)| &=&  \{ (c,d) \in \Div(G)  \st  \lceil (c,d) \rceil \ge 0 \mbox{ and } (c,d) \sim (a,b) \} \\
 &=& \{ (a+np,b-np)  \st n \in \Z, a+np > -1, b-np > -1 \} \p
\end{eqnarray*}

In what follows, we will be brief, and leave most of the details to the reader to verify.
One can check that
$|(a,b)| \ne \emptyset$ if and only if $ \lceil (1+a)/p \rceil + \lceil (1+b)/p \rceil \ge 2$.
The value of $h^0((a,b))$ can be computed as follows:  let $\phi_{p}(x)=\lfloor (x+1)/p \rfloor$, and
\[
h^0((a,b)) = \left\{ \begin{array}{ll}
0 & \mbox{if } \phi_{p}(a) + \phi_{p}(b) < 0 \\
\min\{a+1-p \phi_{p}(a), b+1- p \phi_{p}(b) \} & \mbox{if } \phi_{p}(a) + \phi_{p}(b) = 0 \\
 a+b-p+2 &  \mbox{if } \phi_{p}(a) + \phi_{p}(b) > 0.
\end{array} \right.
\]
Note that for the divisor $D=(0,0)$, we have
\[
h^0((0,0)) = \left\{ \begin{array}{ll} 2-p & \mbox{if } p \le 1 \\ 1 & \mbox{if } p>1 \end{array} \right.
\]
and that the classical inequality $h^0(D) \le \deg(D) +1$ does not hold when $p<1$.

To check that the Riemann-Roch formula holds for a divisor $D=(a,b)$,
it is easiest to consider the three cases for the formula for $h^0((a,b))$.
We note that $(a,b)$ is in one of the three cases if and only if
$(p-2-a,p-2-b)$ is in the opposite case.
It is very straightforward then to check Riemann-Roch in case
$\phi_{p}(a) + \phi_{p}(b) \neq 0$;
one of the two $h^0$ values is zero.
It is a slightly more interesting exercise,
but still straightforward,
to check it in case $\phi_{p}(a) + \phi_{p}(b)=0$.

Unfortunately, the method of direct computation in this example becomes intractable for $R$-graphs with $n>2$.


\section{Change of Rings}

Note that in the definition of the $h^0$ of a divisor,
the minimum is taken over all non-negative $R$-divisors.
Therefore, a priori, the definition of $h^0$ depends on the subring $R$.
We note that if $R \subset S \subset \R$ are two subrings of $\R$,
then any $R$-graph $G$ and $R$-divisor $D$ on $G$
is also an $S$-graph and an $S$-divisor.
In this section we will see that the $h^0$ in fact does not depend on the subring.

Any $H \in \PDiv(G)$ can be written
as an integer linear combination of any $n-1$ elements of the set 
$\{ H_1, H_2, \ldots H_n \}$.
If we exclude $H_k$, for example,
then there are $n-1$ integers $\{m_j\}_{j \ne k}$ such that
$H=\sum_{j \ne k} m_j H_j$,
and we can write $H = \sum_{i=1}^n h_i \cdot v_i $
where 
\begin{equation} \label{hi}
h_i = \left\{ \begin{array}{cl}
m_i \deg(v_i) - \sum_{j \ne k,i} m_j p_{ij} & \mbox{if } i \ne k \\
- \sum_{j \ne k} m_j p_{jk} & \mbox{if } i=k.
\end{array} \right. 
\end{equation}

Let $P_k$ be the $(n-1) \times (n-1)$ matrix
obtained by deleting the $k$th row and column from the matrix $P$.
We can write the $h_i$'s other than $h_k$ in matrix form as
$\textbf{h} = P_k \textbf{m}$
where $\textbf{h}=(h_i)_{i \ne k}$ and $\textbf{m}=(m_i)_{i \ne k}$
are the corresponding column vectors.

For any $\textbf{x}=(x_i) \in \R^{n-1}$ and $c \in \R$,
we say $\textbf{x} \ge c$ if and only if $x_i \ge c$ for each $i$;
similarly for a matrix $A=(a_{ij})$,
we write $A \ge c$ if and only if $a_{ij} \ge c$ for each $i,j$.

A $(n-1) \times (n-1)$ matrix $M$ is \emph{monotone} if $M \textbf{x} \ge 0$ implies that $\textbf{x} \ge 0$
for all $\textbf{x} \in \R^{n-1}$; if $M$ is monotone, it follows that $M$ is nonsingular, with $M^{-1} \ge 0$
(see Chapter 6 in \cite{BP}).  
\begin{lemma} \label{monotone}
$P_k$ in monotone.
\end{lemma}
\begin{proof}
Let $V_i = \{ i' \st p_{ii'} > 0, i' \ne k, i' \ne i \}$
be the set of indices of vertices connected to $v_i$ (excluding $k$).
Suppose that it is the case that $x_i < 0$,
and that $x_i \leq x_{i'}$ for all $i' \in V_i$.
Then
\begin{eqnarray*}
(P_k \textbf{x})_i &=& x_i \deg(v_i) - \sum_{i' \in V_i} x_{i'} p_{ii'} \\
    &=& x_i p_{ik} + x_i \sum_{i' \in V_i} p_{ii'} - \sum_{i' \in V_i} x_{i'} p_{ii'} \\
    &=& x_i p_{ik} + \sum_{i' \in V_i}p_{ii'} (x_i - x_{i'}),
\end{eqnarray*}                           
and we note that with our assumptions, no term here is positive.
Since the sum is non-negative, we conclude that all terms are zero.
We have verified the following therefore, if $P_k \textbf{x} \geq 0$:
\begin{equation} \label{connected}
x_{i} < 0 \mbox{ and } x_{i} \leq x_{i'} \mbox{ for all } i' \in V_{i}
\Rightarrow 
p_{ik}=0 \mbox{ and } x_i = x_{i'} \mbox{ for all } i' \in V_i.
\end{equation}

Now assume that $\textbf{x} \ngeq 0$;
then there is an index $j$ such that $x=x_j<0$ and $x_j \le x_i$ for all $i \ne k$.
By (\ref{connected}), we conclude that $x_i = x$ for all $i \in V_j$,
and also that $p_{jk} = 0$.
We see, by induction on the distance in $G$ to the vertex $v_j$,
that we must have $x_i= x$ and $p_{ij}=0$ for all $i \neq k$.
This contradicts the connectedness of $G$: vertex $v_k$ has no edges on it.
Thus $P_{k}$ is monotone.

\end{proof}

We can now prove the main result for this section.

\begin{proposition}\label{changeofring}
Suppose that all of the entries of the matrix $P$ are in two subrings $R$ and $R'$,
and that all the coordinates of the divisor $D$ are also in both $R$ and $R'$.
Then (using the obvious notation)
$h^0 = h^{0'}$.
\end{proposition}

\begin{proof} 
It suffices to prove the statement when one of the subrings is $R$ and the other is $\R$.
In this case we'll use the notation $Rh^0$ and $\R h^0$, respectively,
for the two minima in question.

First note that the linear system $|D|$ is clearly independent of the ring;
and in particular, whether a linear system is empty or not is also independent.

Therefore, the minimum in question for the $\R h^0$ computation
is over a strictly larger set of divisors;
and hence there can only be a smaller minimum.
This proves that $R h^0(D) \geq \R h^0(D)$.

Suppose that $E$ is an $\R$-divisor, $E \geq 0$, and $|D-E| = \emptyset$,
achieving the minimum, so that $\R h^0(D) = \deg(E)$.
If $E$ is an $R$-divisor,
it also achieves the minimum in $R$ and $R h^0(D) = \R h^0(D)$.
We will show that in fact $E$ must be an $R$-divisor.

Now suppose that $E$ is not an $R$-divisor,
and write $D=\sum_{i=1}^n d_i \cdot v_i$ 
and $E=\sum_{i=1}^n e_i \cdot v_i$,
with $k$ the index of an element such that $e_k \notin R$.
Since $\R h^0(D)=\deg(E)$,
for any $\epsilon \in \R$ with $0 < \epsilon \leq e_k$,
we have that $E-\epsilon \cdot v_k \geq 0$, and therefore
$|D-E + \epsilon \cdot v_k| \ne \emptyset$.
Hence there are principal divisors $H$
such that $D-E+\epsilon\cdot v_k + H > -1$.

Let $\mathcal{H}_\epsilon$ be the set of all such $H$;
by assumption, this is a nonempty set.
Note that if $H \in \mathcal{H}_\epsilon$,
and $H = \sum_{i=1}^n h_i \cdot v_i$,
then $d_i - e_i + h_i  > -1$ for each $ i \ne k$,
and
\begin{equation}\label{dkek}
d_k - e_k + \epsilon + h_k  > -1.
\end{equation}
Also, since $|D-E|=\emptyset$,
there is a $k'$ such that $d_{k'} -e_{k'} +h_{k'} \le -1$;
combined with the conditions above,
the only possibility is $k'=k$. 
Since $d_k \in R$, $h_k \in R$ and $e_k \notin R$,
$d_k-e_k+h_k \ne -1$, and thus $d_k -e_k + h_k < -1$.
Hence $-1-\epsilon <  d_k -e_k + h_k < -1$.

For any $H \in \mathcal{H}_\epsilon$,
there are unique integers $m_i$ such that $H=\sum_{i \ne k} m_i H_i$.
Let
$\textbf{d} = (d_i)_{i \ne k}$,
$\textbf{e} = (e_i)_{i \ne k}$, and
$\textbf{m} = (m_i)_{i \ne k}$
be the corresponding column vectors, and define
$ \textbf{f} = (f_i)_{i \ne k} = \textbf{d}-\textbf{e}+P_k \textbf{m} $.
Note that $\textbf{f} > -1$,
and $h_k = -\sum_{i\neq k} m_k p_{ik}$ by (\ref{hi}).

We can write
$ \textbf{m} = P_k^{-1} ( \textbf{f} - \textbf{d} + \textbf{e})$,
and by Lemma \ref{monotone}, $P_k^{-1} \ge 0$.  
Therefore, since $\textbf{e} \ge 0$ and $\textbf{f} > -1$,
the $m_i$ are bounded from below; set $M \le m_i$ for all $i \ne k$.

We claim that, for $H=\sum_{i \ne k} m_i H_i \in \mathcal{H}_\epsilon$,
the possible coordinates $h_k = -\sum_{i \neq k} m_k p_{ik}$
form a discrete set.
It will suffice to show that, for any real $x$,
the possible coordinates $h_k$
which are at least $-x$ is a finite set.

To that end, for any $x \in \R$ set 
$\mathcal{H}_\epsilon(x) =
\{ H \in \mathcal{H}_\epsilon \st \sum_{i \ne k} m_i p_{ik} \le x \}$;
for large enough $x$ this set is nonempty.

Fix $x\in \R$ such that $\mathcal{H}_\epsilon(x) \ne \emptyset$
and choose $j \ne k$ such that $p_{jk} > 0$.
For $H =\sum_{i \ne k} m_i H_i \in \mathcal{H}_\epsilon(x)$ we then have 
\[
M \le m_j \le
\frac{x - \sum_{i  \ne j,k} m_i p_{ik}}{p_{jk}} \le
\frac{x - M \sum_{i \in V_k,i \ne j} p_{ik}}{p_{jk}}.
\]
Thus the coefficients $m_j \in \Z$ are bounded both below and above,
and hence can take on only finitely many values.
It follows that the set of possible values of
$h_k = - \sum_{i \ne k} m_i p_{ik}$ is also finite,
for $H\in\mathcal{H}_\epsilon(x)$.
As noted above, this implies that these coordinates $h_k$,
for $H\in\mathcal{H}_\epsilon$, form a discrete set.
This in turn implies that there is a maximum value $h$ for the possible $h_k$,
since for all such we have $d_k -e_k + h_k < -1$.

Note that if $\epsilon < \epsilon'$,
then $\mathcal{H}_\epsilon \subset \mathcal{H}_{\epsilon'}$.

We may now shrink $\epsilon$ (if necessary) to achieve $\epsilon < e_k-d_k-h-1$.
This gives a contradition,
since now $d_k-e_k+\epsilon + h_k \leq d_k-e_k+\epsilon + h < -1$
for $H \in \mathcal{H}_\epsilon$,
violating (\ref{dkek}).
We conclude that $E$ is in fact an $R$-divisor as desired,
finishing the proof.

\end{proof}

The result above allows us to simply consider the case of $\R$-graphs.

At the other end of the spectrum, the case of $\Z$-graphs is
equivalent to the Baker-Norine theory.

The Baker-Norine dimension
of a linear system associated with a divisor $D$ on a graph $G$
defined in \cite{BN}
is equal to 
\[
r(D) = \min\{ \deg(E) \st
E \in \Div(G), E \ge 0 \mbox{ and } |D-E|_{BN} = \emptyset \} -1
\]
where here the linear system associated with a divisor $D$ is
\[
|D|_{BN} = \{ D' \in \Div(G) \st D' \ge 0 \mbox { and } D \sim D' \} \p
\]
If we are restricted to $\Z$-divisors on $\Z$-graphs,
the $h^0$ dimension is compatible with the Baker-Norine dimension:

\begin{lemma} \label{LBN}
If $G$ is a $\Z$-graph and $D$ a $\Z$-divisor on $G$, then $h^0(D)=r(D)+1$.
\end{lemma}

\begin{proof}
Note that $\lceil D \rceil = D$ since each component of $D$ is in $\Z$.
This implies that $|D|=|D|_{BN}$ which gives the result.               
\end{proof}


\section{Reduction to $\Q$-graphs}

Note that the definition of $h^0(D)$ depends on the coordinates of $D$
and on the entries of the matrix $P$ which give the edge-weights of the graph $G$.
Indeed,
the set $\mathcal{E}$
of divisors with empty linear systems
depends continuously on $P$, as a subset of $\R^n$.
(If $\mathcal{F}_0$ is the set of divisors $D$ with $d_i > -1$ for each $i$, $\mathcal{E}$
is the complement of the union of all the translates of $\mathcal{F}_0$ by the columns of $P$.)

\begin{proposition}\label{QRProposition}
Suppose that the Riemann-Roch Theorem \ref{RRTheorem} is true for connected $\Q$-graphs.
Then the Riemann-Roch Theorem is true for connected $\R$-graphs.
\end{proposition}
\begin{proof}
Assume that \ref{RRTheorem} holds when $G$ is a $\Q$-graph and $D$ is a $\Q$-divisor.

Suppose that $G$ is a $n$-vertex $\R$-graph and $D$ a $\R$-divisor on $G$. 
Choose any $\epsilon \in \R$ such that $\epsilon >0$.
Since $\Q$ is dense in $\R$, we can choose a $\Q$ divisor $D'$ such that 
\[ 
0 \le D'(v_{i})-D(v_{i}) < \frac{\epsilon}{n}
\] 
each each vertex $v_{i}$ of $G$.  Similarly, we can choose nonnegative edge-weights $p'_{ij} \in \Q$ such that
\[
|p_{ij} -p'_{ij}| < \frac{2\epsilon}{n(n-1)}
\] 
for each $1 \le i<j \le n$, which defines a $\Q$-graph $G'$.  
Let $\deg(v'_{j}) = \sum_{i \ne j} p'_{ij}$ be the degree of the $j$th vertex of $G'$, and set
$K'(v_{i}) = \deg(v'_{i}) -2$ and $g' = \sum_{i<j} p'_{ij} - n + 1$.  We then have
$0 \le \deg(D')-\deg(D) < \epsilon$
and 
$|g'-g|< \epsilon$.

From the definition of $h^{0}$ if follows that 
there is a $\R$-divisor $E$ such that $D-E \ge 0$ with $\deg(D-E)=d$ and $|E|=\emptyset$. 
Since $h^{0}(D)$ varies continuously with the coordinates of $D$, 
it follows from Corollary~\ref{changeofring} that since $\deg(D'-E) -\deg(D-E) < \epsilon$,
$h^{0}(D')-h^{0}(D) < \epsilon$.  Similarly, $|h^{0}(K'-D') - h^{0}(K-D)| < 2 \epsilon$.

Since \ref{RRTheorem} holds for $D'$, 
\[ 
h^{0}(D') - h^{0}(K'-D') - \deg(D') - 1  + g' = 0,
\]
which implies that 
\[ 
|h^{0}(D) - h^{0}(K-D) - \deg(D) -1 + g | < 5 \epsilon.
\]
Since $\epsilon$ was arbitrary, we have
\[
h^{0}(D) - h^{0}(K-D) - \deg(D) -1 + g = 0.
\]

\end{proof}


\section{Scaling}

Suppose that $G$ is an $R$-graph, with edge weights $p_{ij}$.
For any $a>0$, $a \in R \subset \R$,
define $aG$ to be the $R$-graph with the same vertices, and edge weights $\{ap_{ij}\}$.
In other words, if $P$ defines $G$, then $aG$ is the $R$-graph defined by the matrix $aP$.

We will use subscripts to denote which $R$-graph we are using to compute with, e.g.,
$|D|_G$, $h^0_G(D)$, etc. if necessary.
 
For any divisor $D$ on $G$ and $a > 0$, define
\[
T_a(D) = aD + (a-1)I
\]
where 
\[
I= \sum_i 1 \cdot v_i \p
\]

The transformation $T_a$ is a homothety by $a$, centered at $-I$.

\begin{lemma} \label{L1}
Let $D$ be an $R$-divisor.  If $a,b>0$ with $a,b \in R$, then the following hold:
\begin{enumerate}
\item $T_b\circ T_b  = T_{ab}$
\item $T_a(D+H) = T_a(D) + aH$
\item $\lceil D \rceil \geq 0  \Leftrightarrow  \lceil T_a(D)) \rceil \geq 0$
\item $|D|_G \neq \emptyset  \Leftrightarrow |T_a(D)|_{aG} \neq \emptyset$
\item $|D-E|_G \neq \emptyset \Leftrightarrow |T_a(D)-aE|_{aG} \neq \emptyset$
\end{enumerate}
\end{lemma}

\begin{proof} \hspace{2cm}
\begin{enumerate}
\item Suppose that $D = \sum_i d_i \cdot v_i$. Then:
\begin{eqnarray*}
T_a(T_b(D)) &=& T_a \left( \sum_{i} (b d_i +b-1)\cdot v_i \right) \\
                     &=& \sum_{i} (a(b d_i + b -1) + a-1) \cdot v_i \\
                     &=&  \sum_{i} (ab d_i + ab -a + a-1) \cdot v_i \\
                     &=& \sum_{i} (ab d_i + ab -1) \cdot v_i  \\
                     &=& T_{ab}(D) \p
\end{eqnarray*}                      
\item Let $a>0$ and $D,H \in \Div(G)$, then
\begin{eqnarray*}
T_a(D+H) &=&  a(D+H) + (a-1)I \\
                 &=& aD+ aH + (a-1)I \\
                 &=& T_a(D) + aH \p
\end{eqnarray*}   
\item Let $D=\sum_{i} d_i \cdot v_i \in \Div(G)$ and $a > 0$.
Since 
$T_a(D) = \sum_{i} ( a d_i +a-1) \cdot v_i$, we have
\begin{eqnarray*}
\lceil T_a(D)) \rceil \geq 0  & \Leftrightarrow & ad_i+a-1 > -1 \mbox{ for each } i \\
                              & \Leftrightarrow & d_i > -1 \mbox{ for each } i\\
                              & \Leftrightarrow & \lceil D \rceil \geq 0 \p
\end{eqnarray*}
\item Suppose $|D|_G \neq \emptyset$.
Then there is a $H \in \PDiv(G)$
such that $\lceil D+H\rceil \geq 0$.
Since $T_a(D+H)= T_a(D)+aH$ and $aH \in \PDiv(aG)$, 
by part (3)
we have $\lceil T_a(D)+aH \rceil \geq 0$ and thus $|T_a(D)|_{aG} \neq \emptyset$.

The converse is an identical argument.

\item Let $D'=D-E$; then from (4),
$ |D'|_G \neq \emptyset \Leftrightarrow | T_a(D')|_{aG}   \neq \emptyset $
where $T_a(D') = T_a(D-E) = T_a(D)-aE$.

\end{enumerate}
\end{proof}

\begin{corollary} \label{C1}
$ h^0_{aG}(T_a(D)) = ah^0_G(D)$
\end{corollary}

\begin{proof}
Since $a>0$, from Lemma~\ref{L1} (5) we have
\begin{eqnarray*}
h^0_{aG}(T_a(D)) &=&
\min_{E' \in \Div(aG)}\{ \deg(E') \st E' \ge 0, |T_a(D) - E'|_{aG} = \emptyset \} \\
   &=& \min_{E \in \Div(G)}\{ \deg(aE) \st aE \ge 0, |T_a(D) - aE|_{aG} = \emptyset \} \\
   &=& a \left( \min_{E \in \Div(G)}\{ \deg(E) \st E \ge 0, |T_a(D) -aE |_{aG} = \emptyset \} \right)  \\
   &=& a \left( \min_{E \in \Div(G)}\{ \deg(E) \st E \ge 0, |D -E |_{G} = \emptyset \} \right) \\
   &=& a h^0_G(D) \p
\end{eqnarray*}
\end{proof}

\begin{lemma} \label{L2}
Let $D$ be an $R$-divisor.  If $a>0$ with $a \in R$ then the following hold:
\begin{enumerate}
\item $K_{aG} = T_a(K_G) + (a-1)I$
\item $K_{aG} - T_a(D) = T_a(K_G - D)$
\item $\deg(T_a(D)) = a \deg(D) + (a-1)(n)$
\item $g_{aG} = ag_G + (a-1)(n-1)$.
\end{enumerate}
\end{lemma}

\begin{proof} \hspace{2cm}
\begin{enumerate}
\item Since $K_{aG} = \sum_{i} (a \deg(v_i) -2) \cdot v_i$, we have
\begin{eqnarray*}
T_a(K_G) &=& T_a(\sum_{i} (\deg(v_i)-2) \cdot v_i) \\
                &=& a\sum_{i} (\deg(v_i)-2) \cdot v_i  +  \sum_{i} (a - 1) \cdot v_i \\
                &=& \sum_{i} (a \deg(v_i) -2a +a - 1 ) \cdot v_i \\
                &=& \sum_{i} (a \deg(v_i) -a - 1 ) \cdot v_i \\                
                &=& K_{aG} - (a-1) I \p
\end{eqnarray*}              
\item \begin{eqnarray*}
K_{aG} - T_a(D) &=& T_a(K_G) + (a-1)I - T_a(D) \\
                           &=& aK_G + (a-1)I + (a-1)I -aD -(a-1)I \\
                           &=& a(K_G - D) + (a-1)I \\
                           &=& T_a(K_G - D) \p
\end{eqnarray*}       
\item \begin{eqnarray*}
\deg(T_a(D)) &=& \deg(aD + (a-1)I) \\
                      &=& a \deg(D) + (a-1) \deg(I) \\
                      &=& a \deg(D) + (a-1)(n) \p
\end{eqnarray*}        
\item \begin{eqnarray*}
g_{aG} &=& \sum_{i} ap_{ij} - n + 1 \\
             &=& a \sum_{i} p_{ij} - an + a + (a-1)n + 1 - a\\
             &=& a g_G + (a-1)(n-1) \p
\end{eqnarray*}     
\end{enumerate}  
\end{proof}


\section{Reduction to $\Z$-graphs}

\begin{theorem} \label{T1}
Let $a>0$; then 
\begin{equation} \label{RR1}
h^0_G(D) - h^0_G(K_G -D) = \deg(D) -g_G +1
\end{equation}
if and only if 
\begin{equation} \label{RR2}
h^0_{aG}(T_a(D)) - h^0_{aG}(K_{aG} -T_a(D)) = \deg(T_a(D)) -g_{aG} +1 \p 
\end{equation}
\end{theorem}

\begin{proof}
Let $a>0$.  Multiplying (\ref{RR1}) by $a$, we have
\[
ah^0_G(D) - ah^0_G(K_G -D) = a\deg(D) -ag_G +a .
\]
The left side of this equation is equal to
\[
h^0_{aG}(T_a(D)) - h^0_{aG}(T_a(K_G-D))
= h^0_{aG}(T_a(D)) - h^0_{aG}(K_{aG}-T_a(D))
\]
using Corollary \ref{C1} and Lemma \ref{L2} (2).
The right side of the equation is
\[
\deg(T_a(D)) - (a-1)(n) - g_{aG} + (a-1)(n-1) +a
= \deg(T_a(D))  -g_{aG} + 1
\]
using Lemma \ref{L2} (3) and (4).
This proves that (\ref{RR1}) implies (\ref{RR2});
the converse is identical.
\end{proof}

\begin{corollary}\label{ZQCorollary}
Suppose that the Riemann-Roch Theorem \ref{RRTheorem} is true for connected $\Z$-graphs.
Then the Riemann-Roch Theorem is true for connected $\Q$-graphs.
\end{corollary}

\begin{proof}
Given a connected $\Q$-graph $G$ and a $\Q$-divisor $D$ on it,
there is an integer $a > 0$ such that
$aG$ is a connected $\Z$-graph and $T_a(D)$ is a $\Z$-divisor.
Therefore by hypothesis, the Riemann-Roch statement (\ref{RR2}) holds.
Hence by Theorem \ref{T1}, (\ref{RR1}) holds, 
which is the Riemann-Roch theorem for $D$ on $G$.
\end{proof}

We now have the ingredients to prove Theorem \ref{RRTheorem}.
\begin{proof}
First, we note again that the Riemann-Roch Theorem of \cite{BN}
is equivalent to the Riemann-Roch theorem for connected $\Z$-graphs
in our terminology.
Therefore, using Corollary \ref{ZQCorollary}, we conclude that
the Riemann-Roch Theorem is true for connected $\Q$-graphs.
Then, using Proposition \ref{QRProposition}, we conclude
that Riemann-Roch holds for connected $\R$-graphs.

Finally, Proposition \ref{changeofring} finishes the proof
of the Riemann-Roch theorem for divisors on arbitrary $R$-graphs,
for any subring $R \subset \R$.
\end{proof}


\begin{thebibliography}{99}

\bibitem{BN}
Baker, Matthew  and Norine, Serguei,
Riemann-Roch and Abel-Jacobi Theory on a Finite Graph,
\emph{Advances in Mathematics} 215, 2007, 766-788.

\bibitem{BP}
Berman, Abraham and Plemmons, Robert J.,
\emph{Nonnegative Matrices in the Mathematical Sciences},
Classics in Applied Mathematics,
SIAM, Philadelphia, PA, 1994.

\bibitem{GK}
Gathmann, Andreas and Kerber, Michael,
A Riemann-Roch Theorem in Tropical Geometry,
\emph{Mathematische Zeitschrift} 259, 2008, 217-230.

\bibitem{MZ}
Mikhalkin, Grigory and Zharkov, Ilia,
Tropical Curves, Their Jacobians, and Theta Functions,
\emph{preprint} arXiv:math/0612267v2 [math.AG], 20 Nov 2007.

\end{thebibliography}
\end{document}